\documentclass[letterpaper, 10 pt]{IEEEtran}
\IEEEoverridecommandlockouts

\makeatletter
\newcommand{\subalign}[1]{%
	\vcenter{%
		\Let@ \restore@math@cr \default@tag
		\baselineskip\fontdimen10 \scriptfont\tw@
		\advance\baselineskip\fontdimen12 \scriptfont\tw@
		\lineskip\thr@@\fontdimen8 \scriptfont\thr@@
		\lineskiplimit\lineskip
		\ialign{\hfil$\m@th\scriptstyle##$&$\m@th\scriptstyle{}##$\hfil\crcr
			#1\crcr
		}%
	}%
}
\makeatother

\usepackage[utf8]{inputenc}

\usepackage{amsmath}
\usepackage{amssymb}

\usepackage{amsthm}

\theoremstyle{theorem}
\newtheorem{thm}{Theorem}[section]

\newtheorem{lemma}{Lemma}[section]

\theoremstyle{definition}
\newtheorem{controlstrategy}{Control Strategy}[section]

\newtheorem{assumption}{Assumption}[section]
\newtheorem{definition}{Definition}[section]

\usepackage{graphics}
\usepackage[font=small]{caption}
\usepackage{subcaption}

\usepackage{todonotes}

\newcommand{\pars}{\pi}

\newcommand{\R}{\mathbb{R}}
\newcommand{\C}{\mathbb{C}}

\newcommand{\unders}{\underline s}

\newcommand{\overs}{\overline s}

\newcommand{\Ms}{\mathcal M^{\rm s}}
\newcommand{\Mu}{\mathcal M^{\rm u}}

\DeclareMathOperator{\sgn}{sgn}
\DeclareMathOperator{\Real}{Re}

\usepackage{xcolor}
\newcommand{\RM}[1]{{#1}}

\usepackage{hyperref}

\title{\LARGE \bf
Control of multistability through local sensitivity analysis: \\ application to cellular decision-making networks\thanks{This work was supported by DGAPA-PAPIIT(UNAM) grant n. IN102420 and Conacyt CB grant n. A1-S-10610.}
}

\author{Rodrigo Moreno-Morton$^{\dagger1}$ and Alessio Franci$^{\dagger2}$
\thanks{$\dagger$ Math Department at the Faculty of Sciences, National Autonomous University of Mexico (UNAM). $^{1}${\footnotesize luiromormor@gmail.com}. $^{2}${\footnotesize afranci@ciencias.unam.mx}}
}

\begin{document}

\maketitle
\thispagestyle{empty}
\pagestyle{empty}

\begin{abstract}
	Control of multistable dynamics has important applications, from physics to biology \RM{but the complexity of the systems of differential equations used for their modeling often makes this problem intractable from a global perspective. Here, we propose that for a certain class of multi-stable dynamical systems, including monotone systems, linearized control at the stable and saddle points of the multi-stable dynamics can lead to predictable global changes in the relative sizes and depths of its basins of attraction.} Our parameter control signal is computationally cheap and provides counter-intuitive information about the sensitive parameters to be manipulated in an experimental setting.
\end{abstract}

\section{Introduction}

Multistable dynamics and their control  appear in a variety of physical, engineered, and biological systems~\cite{pisarchik2014control}. Opinion-formation and decision-making networks also exhibit multistability and are receiving increasing attention due to their relevance in sociopolitical systems and for collective behaviors in both biological and artificial agent groups. In these networks, each attractor corresponds to a decision state~(see~\cite{bizyaeva2022} and references therein). Another important example of multistable decision-making are molecular regulatory networks, where decisions correspond to cellular phenotypes~\cite{balazsi2011cellular}. Cellular decision-making can be functional, e.g., development, but also deleterious for the organism, e.g., tumor formation.

A well studied case of cellular decision-making with both functional or deleterious consequences is the epithelial-mesenchymal transition (EMT), a change in phenotype of epithelial cells from a clearly polarized, epithelium-adhered cell (epithelial phenotype) to a non-polarized, free-moving cell (mesenchymal phenotype)~\cite{nieto_emt_2016, yang_epithelial-mesenchymal_2008}. EMT plays an important role in several stages of embryonic development and tissue repair but also in tumor metastasis~\cite{chakrabarti_elf5_2012}. Understanding how the EMT is regulated and how it may be controlled is relevant both for biology and medicine.


Global analysis and control of multi-stable dynamics are hard in general. Although almost global Lyapunov functions are known to exists for multistable systems~\cite{efimov2012global}, their analytical computation is impossible except for simple low-dimensional examples.  Some local approaches, based on the model linearization, have been developed for discrete-time dynamical systems with chaotic dynamics~\cite{pisarchik2014control}.

Motivated by EMT control, in this paper we focus on multi-stable dynamics in which the only attractors are exponentially stable equilibria and in which the boundaries of the basins of attraction are the stable manifolds of saddle points. \RM{Monotone dynamical systems~\cite{smith1988systems} in any dimension are important representatives of this class of systems}.

\RM{The paper contributions are the following. Under suitable monotonicity assumptions, we prove the existence of a simple relation between the stability margin of an equilibrium of a one-dimensional multistable dynamical system and the size of its basin of attraction. We single out a class of multistable dynamical systems in arbitrary dimension for which the one-dimensional theory suggests a simple strategy to control the relative size and depth of the various basins of attraction. Although grounded on heuristic arguments, our control strategy solely uses the local sensitivity of the linearized model dynamics at the relevant equilibria. In particular, it does not require the knowledge of an analytic expression for the model almost-global Lyapunov function. Its effectiveness is illustrated on a simple two-dimensional monotone dynamical system and on a novel four-dimensional model of the EMT.}

\RM{The one-dimensional theory is illustrated in Section~\ref{sec: 1d theory}. The relevant class of high-dimensional dynamics and the proposed strategy for their multistability control are introduced in Section~\ref{sec: high dim}. Applications to a two-dimensional monotone dynamics and to a four-dimensional EMT model are illustrated in Section~\ref{sec: 2d example} and Section~\ref{sec: MET control}. Limitations are discussed in Section~\ref{sec: conclusions}. Section~\ref{SEC: code} provides the project GitHub page.}

\section{Notation and Definitions}
\label{sec: not def}

$\R$ denotes the set of real numbers. $\mathbb N$ denotes the set of positive integers. $\mathbb C$ denotes the set of complex numbers. $\langle\cdot,\cdot\rangle:\R^n\times\R^n\to\R$ denotes the Euclidean product in $\R^n$ and $\|\cdot\|:\R^n\to\R$ its induced norm. Given $x\in\R$, $\sgn(x)$ denotes its sign, i.e., $\sgn(x)=-1$ if $x<0$, $\sgn(x)=0$ if $x=0$, and $\sgn(x)=1$ if $x>0$. Given $z\in\mathbb C$, $\Real(z)$ denotes its real part. A set $\mathcal K\subset\R^n$ is a linear cone if for all $x\in\mathcal K$ and all $a>0$, $ax\in\mathcal K$. \RM{An orthant $K\subset\R^n$ is the cone $\{x\in\R^n:(-1)^{m_i}x_i\geq 0,\,i=1,\ldots,n,\, m_i\in\{0,1\}\}$.}

Consider a smooth dynamical system
\begin{equation}\label{eq: general dyn syst}
	\frac{dx}{dt}=:\dot x=f(x),\quad f:\R^n\to\R^n.
\end{equation}
with flow $\varphi:\R^n\times\R\to\R^n$, i.e., given $y\in\R^n$, $x(t)=\varphi(y,t)$ is the solution of~\eqref{eq: general dyn syst} at time $t$ with initial conditions $x(0)=y$.
An equilibrium $x^*$ of~\eqref{eq: general dyn syst} is called hyperbolic if the Jacobian $\frac{\partial f}{\partial x}(x^*)\in\mathbb\R^{n\times n}$ has no eigenvalues on the imaginary axis. The stable manifold of an equilibrium $x^*$ is the set
$\left\{x\in\R^n\,:\, \lim_{t\to\infty}\varphi(x,t)=x^*\right\}$. The unstable manifold of an equilibrium $x^*$ is the set
$\left\{x\in\R^n\,:\, \lim_{t\to-\infty}\varphi(x,t)=x^*\right\}$. If a hyperbolic equilibrium $x^*$ is stable, that is, if $\frac{\partial f}{\partial x}(x^*)$ has only eigenvalues with negative real part, then its stable manifold is called its basin of attraction. A set $U\subset\R^n$ is called invariant for model~\eqref{eq: general dyn syst} if $x\in U$ implies $\varphi(x,t)\in U$ for all $t\in\R$. \RM{A trajectory is called heteroclinic if $\lim_{t\to-\infty}\varphi(x,t)=x_1^*$ and $\lim_{t\to\infty}\varphi(x,t)=x_2^*$ for two distinct equilibria $x_1^*,x_2^*$. Given a linear cone $\mathcal K$, a trajectory $\varphi(x,t)$ is called $\mathcal K$-monotone if $\varphi(x,t_2)-\varphi(x,t_1)\in\mathcal K$ for all $t_2>t_1$.}

\section{One-dimensional theory}
\label{sec: 1d theory}

Consider a one-dimensional dynamical system
\begin{equation}\label{eq: generic 1D}
	\dot x = f(x) = -F'(x)
\end{equation}
where $f:\R\to\R$ is smooth and $F$ is a primitive of $-f$, i.e., $F(x)=-\int_0^x f(y)dy+C$, $C\in\R$. $F(x)$ is an almost global Lyapunov function for~\eqref{eq: generic 1D}, in the sense of~\cite{efimov2012global}. Suppose that $x_0$ is an exponentially stable equilibrium of~\eqref{eq: generic 1D}, i.e.,
\begin{equation}\label{eq: stable eq 1D}
	F'(x_0)=0,\quad F''(x_0)>0.
\end{equation}
Generically there are three possible cases for the geometry of the basin of attraction of $x_0$: unbounded (Figure~\ref{fig: basins illustration}a), half-bounded (Figure~\ref{fig: basins illustration}b,c), bounded (Figure~\ref{fig: basins illustration}d). Boundaries of the basin of attraction, when they exist, are given by unstable points $\unders_0$, $\overs_0$, with $\unders_0 < x_0 < \overs_0$, satisfying \footnote{We denote the unstable equilibria with an $s$ because in higher dimension they will correspond to {\it saddle} points.}
\begin{equation}\label{eq: unstable eqs 1D}
	F'(\unders_0)=F'(\overs_0)=0,\quad F''(\unders_0),F''(\overs_0)<0.
\end{equation}
We focus on the bounded case (Figure~\ref{fig: basins illustration}d) but we also remark that our theoretical developments apply naturally to the bounded side of the half-bounded cases (Figure~\ref{fig: basins illustration}b,c).

\begin{figure}
	\centering
	\includegraphics[width=0.32\textwidth]{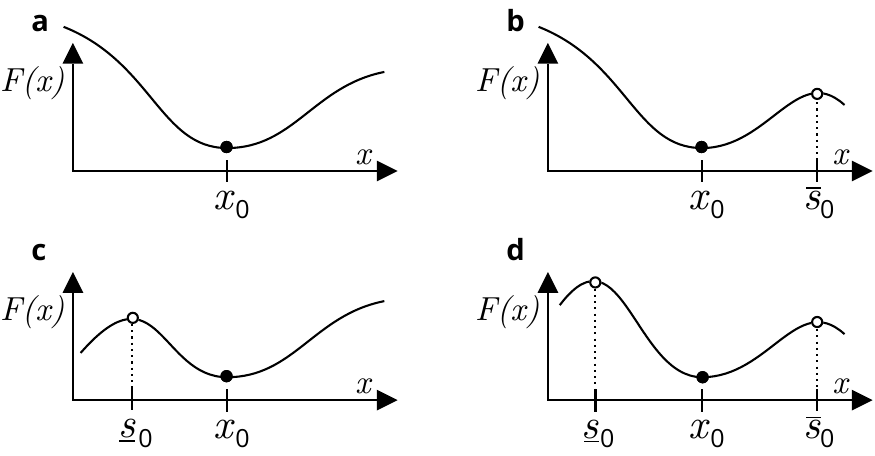}
	\caption{Unbounded vs bounded one-dimensional basins of attraction. {\bf a}. Unbounded. {\bf b},{\bf c}. Half-bounded. {\bf d}. Bounded. Stable equilibria are denoted by dots; unstable equilibria by circles.}\label{fig: basins illustration}
\end{figure}

Our goal is understanding how the presence of possible control parameters in model~\eqref{eq: generic 1D} affect the depth (i.e., local stability) and size of the basin of attraction of $x_0$, as well as understanding how depth and size can be related.

\subsection{Increasing local \RM{stability margins} of $x_0$ increases the global size of its basin of attraction}

Suppose a parameter $\alpha$ tunes the stability of $x_0$ as
\begin{equation}\label{eq: generic 1D lin control}
	\dot x = -F'(x)-\kappa(x,\alpha),
\end{equation}
where $\kappa:\R\times\R\to\R$ is smooth \RM{and satisfies the following {\it monotonicity conditions}}
\begin{equation}\label{eq: stab feedback conditions}
	\begin{aligned}
		&\kappa(x,0)=0,\quad \sgn\left(\frac{\partial\kappa}{\partial x}(x,\alpha)\right)=\sgn(\alpha),\\ &\sgn\left(\frac{\partial\kappa}{\partial \alpha}(x,0)\right)=\sgn(x-x_0).
	\end{aligned}
\end{equation}
As representative examples, $\kappa(x,\alpha)=\alpha(x-x_0)$ or $\kappa(x,\alpha)=\alpha\bar\kappa(x)$ with $\bar\kappa(\cdot)$ strictly monotone increasing and $\bar\kappa(x_0)=0$. Observe that model~\eqref{eq: generic 1D lin control} has almost global Lyapunov function $V(x,\alpha)=F(x)+\int_0^x\kappa(y,\alpha)dy$.

Increasing $\alpha$ \RM{increases the stability margins} of $x_0$ because $	\left.\frac{\partial}{\partial \alpha}\frac{\partial\dot x}{\partial x}\right|_{\subalign{x&=x_0\\\alpha&=0}}<0$, which implies $\left.\frac{\partial\dot x}{\partial x}\right|_{\subalign{x&=x_0\\\alpha&\gtrsim0}}<\left.\frac{\partial\dot x}{\partial x}\right|_{\subalign{x&=x_0\\\alpha&=0}}<0$, i.e., the eigenvalue of the linearization at $x_0$ becomes more negative as $\alpha$ is increased.

Let $\unders(\alpha)$ and $\overs(\alpha)$ be the boundaries of the basin of attraction of $x_0$ for sufficiently small $\alpha$. In particular, $\unders(0)=\unders_0$ and $\overs(0)=\overs_0$.

\begin{lemma}\label{lem: stabilizing increases}
The two functions $\unders(\alpha)$ and $\overs(\alpha)$ are differetiable at $\alpha=0$. Furthermore $\unders'(0)<0$ and $\overs'(0)>0$.
\end{lemma}
\begin{proof}
By definition $V'(\unders(\alpha),\alpha)\equiv 0$ for $\alpha$ sufficiently small. Differentiating with respect to $\alpha$ we get
\begin{equation}\label{eq: unst 1d impl deriv}
	\left(F''(\unders(\alpha))+\frac{\partial\kappa}{\partial x}(\unders(\alpha),\alpha)\right)\unders'(\alpha)+\frac{\partial\kappa}{\partial \alpha}(\unders(\alpha),\alpha)=0.
\end{equation}
Invoking~\eqref{eq: unstable eqs 1D} and~\eqref{eq: stab feedback conditions}, it follows that $\unders'(0)=-\frac{\partial\kappa}{\partial \alpha}(\unders_0,0)/F''(\unders_0)<0$. Similarly, we get $\overs'(0)>0$.
\end{proof}

The following theorem, illustrated in Figure~\ref{fig: stabilizing increases}, is a direct corollary of Lemma~\ref{lem: stabilizing increases}.

\begin{figure}
	\centering
	\includegraphics[width=0.32\textwidth]{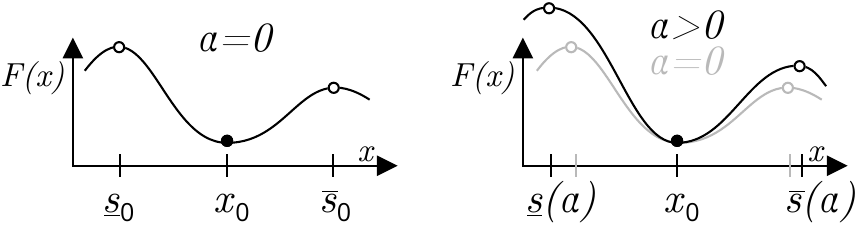}
	\caption{Increasing the stability of an exponentially stable equilibrium increases the size of its basin of attraction, and vice-versa.}\label{fig: stabilizing increases}
\end{figure}

\begin{thm}\label{thm: stabilizing increases}
	In model~\eqref{eq: generic 1D lin control}-\eqref{eq: stab feedback conditions}, enhancing the stability of $x_0$, that is, making $\left.\frac{\partial \dot x}{\partial x}\right|_{x=x_0}$ more negative, by increasing $\alpha$, also increases the size of its basin of attraction.
\end{thm}

\RM{Under suitable assumptions, a simple converse theorem can also be proved, i.e., increasing the size of the basin of attraction of $x_0$ increases its stability margins. Its statement and proof are not included due to space constraints.}

\section{High-dimensional Generalization}
\label{sec: high dim}

Our goal is to apply the results of the one-dimensional theory in Section~\ref{sec: 1d theory} to higher-dimensional systems of the form
\begin{equation}\label{eq: gen high system}
	\dot x=F(x,\pi),
\end{equation}
where $x\in\R^n$, $\pi\in\R^m$ is a vector of parameters, and $F:\R^n\times\R^m\to\R^n$ is smooth.

\begin{definition}\label{def: saddle}
	An unstable equilibrium of model~\eqref{eq: gen high system} is called a {\em saddle point} if its stable manifold is $n-1$-dimensional and an {\em unstable equilibrium} otherwise
\end{definition}

\begin{assumption}\label{ass: separatrices}
	\RM{For all $x_0$, the trajectory $\varphi(x_0,t)$ of~\eqref{eq: gen high system} is bounded and the only attractors of~\eqref{eq: gen high system} are exponentially stable equilibria with real simple eigenvalues. The basins of attraction are separated by the union of the stable manifolds of saddle points and any heteroclinic trajectory from a saddle to a stable equilibrium is $K$-monotone for some orthant $K$.}
\end{assumption}

\RM{
Evidently, Assumption~\ref{ass: separatrices} is satisfied by one-dimensional dynamical systems. Another important class of dynamics satisfying this assumption are monotone dynamical systems with bounded trajectories~\cite[Theorem~2.6 and below]{smith1988systems}. 
}

Under Assumption~\ref{ass: separatrices}, each stable equilibrium $x_0\in\R^n$ of model~\eqref{eq: gen high system} is surrounded by a set of saddle points $s_1,\ldots,s_l\in\R^n$, $l\leq n$, and associated stable and unstable manifolds, $\Ms_{s_i}$ and $\Mu_{s_i}$, $i=1,\ldots,l$, respectively. The union $\bigcup_{i=1}^l {\Ms_{s_i}}$ of the saddle stable manifolds defines the boundary of the basin of attraction of $x_0$. Each saddle unstable manifold $\Mu_{s_i}$ is \RM{the monotone heteroclinic trajectory} connecting $x_0$ and $s_i$ (see Figure~\ref{fig: 1d inv dynamics}). Let
\begin{equation}\label{eq: jacobian}
	J^\pi_x=\frac{\partial F}{\partial x}(x,\pi)
\end{equation}
be the Jacobian of model~\eqref{eq: gen high system} at $x$ \RM{with} parameters $\pi$. Then each $\Mu_{s_i}$ is tangent at $x_0$ to a distinct eigenvector $v_i$ of $J^\pi_{x_0}$ with associated eigenvalue $\lambda_i<0$.

\begin{figure}
	\centering
	\includegraphics[width=0.175\textwidth]{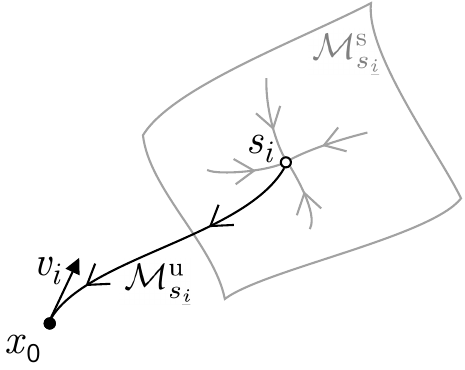}
	\caption{The one-dimensional invariant dynamics over \RM{the $K$-monotone heteroclinic trajectory} $\Mu_{s_i}$ between a stable point $x_0$ and a saddle point $s_i$, and associated eigenvector $v_i$ at $x_0$.}\label{fig: 1d inv dynamics}
\end{figure}

\RM{
\subsection{A heuristic argument.}
	To generalize the one-dimensional theory to model~\eqref{eq: gen high system} we rely on the following argument.
	\begin{itemize}
	    \item[{\it i)}] Multistability in~\eqref{eq: gen high system} is organized by the saddle stable manifolds $\Ms_{s_i}$ and the saddle-to-stable equilibrium heteroclinic orbits (unstable manifolds) $\Mu_{s_i}$.
	    \item[{\it ii)}] Because dynamics on the one-dimensional invariant sets $\Mu_{s_i}$ are monotone, the one-dimensional theory, in particular condition~\eqref{eq: stab feedback conditions}, approximately applies along each $\Mu_{s_i}$. This ensures that pushing the saddle $s_i$ (and therefore its stable manifold $\Ms_{s_i}$) away from a stable equilibrium $x_0$ and decreasing the eigenvalue $\lambda_i$ of $J^\pi_{x_0}$ associated to $v_i$ (as in Figure~\ref{fig: 1d inv dynamics}) are two complementary ways of enlarging the basin of attraction of $x_0$.
	\end{itemize}
	}

\RM{In light of this argument, the proposed generalization involves two control strategies: eigenvalue control and saddle location control. We introduce them in the following sections.}



\subsection{Two useful lemmas}

The following basic lemmas provide the basic tools for eigenvalue and saddle location control.
\begin{lemma}
	\label{lem: equilibrium derivative}
	Let $x^*$ be a hyperbolic equilibrium of~\eqref{eq: gen high system}, i.e., $F(x^*,\pi)=0$. Then
	$
	\frac{\partial x^*}{\partial\pi_i}=-\left(J^\pi_{x^*}\right)^{-1}\frac{\partial F}{\partial\pi_i}(x^*,\pi).
	$
\end{lemma}
\begin{proof}
	Because $x^*$ is hyperbolic, $J^\pi_{x^*}$ is not singular and the lemma follows by the implicit function theorem.
\end{proof}

\begin{lemma}
	\label{lem: eigenvalue derivative}
	Let $x^*$ be a hyperbolic equilibrium of~\eqref{eq: gen high system}. Let $\lambda\in\C$ be an eigenvalue of $J^\pi_{x^*}$ with left eigenvector $w\in\C^n$ and right eigenvector $v\in\C^n$. The total derivative $D_i\lambda = \frac{\partial \lambda}{\partial x}\frac{\partial x^*}{\partial \pars_i} +  \frac{\partial \lambda}{\partial \pars_i}$ of $\lambda$ with respect to $\pi_i$ is given by
	$
	D_i\lambda = \dfrac{w^T D_i J^\pi_{x^*} v}{w^T v}
	$,
	where $D_i J^\pi_{x^*}$ is the total derivative of $J^\pi_{x^*}$ with respect to $\pi_i$, i.e. $D_i J^\pi_{x^*} = \frac{\partial J^\pi_{x^*}}{\partial x}\frac{\partial x^*}{\partial \pars_i} + \frac{\partial J^\pi_{x^*}}{\partial \pars_i}$.
\end{lemma}
\begin{proof}
	By differentiating the equality $w^T J^\pi_{x^*} v=w^T\lambda v$ with respect to $\pi_i$ and noticing that the terms involving derivatives of $w^T$ and $v$ cancel out, we obtain $w^T \left(\dfrac{\partial J^\pi_{x^*}}{\partial x}\dfrac{\partial x^*}{\partial \pars_i} + \dfrac{\partial J^\pi_{x^*}}{\partial \pars_i}\right) v = w^T \left(\dfrac{\partial \lambda}{\partial x}\dfrac{\partial x^*}{\partial \pars_i} +  \dfrac{\partial \lambda}{\partial \pars_i}\right) v$
and the result follows
\end{proof}

\subsection{Eigenvalue control by local sensitivity analysis}

Consider a stable equilibrium $x_0$ of model~\eqref{eq: gen high system} with parameters $\pi$ and let Assumption~\ref{ass: separatrices} hold. Let $\lambda<0$ be an eigenvalue of $J^\pi_{x_0}$ with left eigenvector $w\in\R^n$ and right eigenvector $v\in\R^n$. Our goal is to make $\lambda$ more negative through suitable parameter variation. 
The gradient of $\lambda$ with respect to the vector of parameters can be computed using Lemma~\ref{lem: eigenvalue derivative} as $\nabla\lambda=\left[D_1\lambda,\ldots,D_m\lambda\right]
= \left[\dfrac{w^T D_1 J^\pi_{x^*} v}{w^T v},\ldots,\dfrac{w^T D_m J^\pi_{x^*} v}{w^T v}\right]$,
where $D_i J^\pi_{x^*} = \frac{\partial J^\pi_{x^*}}{\partial x}\frac{\partial x^*}{\partial \pars_i} + \frac{\partial J^\pi_{x^*}}{\partial \pars_i}$ and $\frac{\partial x^*}{\partial \pars_i}$ can be computed as in Lemma~\ref{lem: equilibrium derivative}. The direction in the parameter space that leads to the fastest decrease of $\lambda$ is therefore
\begin{equation}\label{eq: lambda max decrease}
	d_\lambda = -\frac{\nabla\lambda}{\|\nabla\lambda\|}\,.
\end{equation}

Suppose all parameters can be controlled independently and no other parametric constraints are present. The proposed control strategy to decrease $\lambda$ by $\delta\!\lambda>0$ units is the following, where $0<\varepsilon\ll1$ and $n_{\rm ite}\in\mathbb N$ are control hyperparameters.
\begin{controlstrategy}\label{CS: fully actu eig}
	{\it Eigenvalue Control}
	\begin{itemize}
		\item[1.] Compute~\eqref{eq: lambda max decrease} at the current equilibrium and parameter values.
		\item[2.] Update the parameters: $\pi\leftarrow\pi+\varepsilon d_\lambda$ and compute the new equilibrium points.
	    \item[3.] Repeat steps~1 and~2 until $\lambda$ is decreased by $\delta\!\lambda$ or until the maximum number of iterations $n_{\rm ite}$ is exceeded.
	\end{itemize}
\end{controlstrategy}



\subsection{Saddle location control by local sensitivity analysis}

Let $x_0$ be a stable equilibrium of model~\eqref{eq: gen high system} with parameters $\pi$ and let Assumption~\ref{ass: separatrices} hold. Let $s\in\R^n$ be a saddle point whose unstable manifold $\Mu_{s}$ \RM{is heteroclinic} to $x_0$. Let
$g_{s-x_0}(\pi)=\|s-x_0\|$
be the Euclidean distance between the saddle and the stable point as a function of the model parameters. By the chain rule
$
\frac{\partial g_{s-x_0}}{\partial\pi_i}=2(s-x_0)^T\left(\frac{\partial s}{\partial\pi_i}-\frac{\partial x_0}{\partial\pi_i}\right)
$,
where $\frac{\partial s}{\partial\pi_i}$ and $\frac{\partial x_0}{\partial\pi_i}$ can be computed as in Lemma~\ref{lem: equilibrium derivative}.
It follows that the distance between $s$ and $x_0$ is maximally increased along
\begin{equation}\label{eq: saddle dist max increase}
	d_{g_{s-x_0}}=\frac{\nabla g_{s-x_0}}{\|\nabla g_{s-x_0}\|}
\end{equation}
where $\nabla g_{s-x_0}=\left[\frac{\partial g_{s-x_0}}{\partial\pi_1},\ldots,\frac{\partial g_{s-x_0}}{\partial\pi_m}\right]$. Given~\eqref{eq: saddle dist max increase}, we can formulate
\RM{a control strategy to optimally increase $g_{s - x_0}(\pars)$ by $\delta\!g > 0$ units.}

\begin{controlstrategy}\label{CS: fully actu saddle}
	{\it Saddle Control}
	\begin{itemize}
	    \item[1.] Compute~\eqref{eq: lambda max decrease} at the current equilibrium and parameter values.
		\item[2.] Update the parameters: $\pi\leftarrow\pi+\varepsilon d_{g_{s-x_0}}$ and compute the new equilibrium points.
		\item[3.] Repeat Steps~1 and~2 until $g_{s-x_0}(\pi)$ increased by $\delta\!g$ or until the maximum number of iterations $n_{\rm ite}$ is reached.
	\end{itemize}
\end{controlstrategy}


\subsection{Parameter sensitivity}
\label{sec: par sens}

The entries of the parameter control vectors $d_\lambda$ and $d_{g_{s-x_0}}$, provide the sensitivity of the applied control to the parameters. Control is highly sensitive to parameters associated to entries with large absolute value: small variations of those parameters have large effects on the controlled quantity. Control is weakly sensitive to parameters associated to entries with small absolute value: big variations of those parameters are needed to  affect the controlled quantity. In practice, highly sensitive parameters define the targets of experimental manipulations to bring a system to a desired state.

\subsection{Multi-objective and underactuated control}
\RM{
There might be cases in which one needs to consider multiple control objectives at once, e.g., simultaneously increasing the distance between multiple equilibria, decreasing multiple eigenvalues, or any combination of the two control objectives.}
\RM{In such cases, new {\it multi-objective} control strategies based on Control Strategies~\ref{CS: fully actu eig} and~\ref{CS: fully actu saddle} can be implemented using a Multiple-Gradient Descent Algorithm (MGDA)~\cite{desideri2012multiple}. MGDA returns an optimal parameter control direction $\tilde d$ by taking as input the gradients of the different control objectives and by returning a new gradient such that no objective is worsened. This method does not require additional hyper-parameters, e.g., the weights with which different gradients are weighted in the optimization procedure.
}

\RM{Another important case is when certain parameter manipulations are hard or impossible to achieve in practice, in which case we can introduce regularizing cost functions that penalize variations along those parameters or simply project those directions to zero before applying a gradient step. In this case, the resulting control strategy is called {\it underactuated}.}

\subsection{A Two-dimensional Monotone Example}
\label{sec: 2d example}

Consider the following two-dimensional dynamics
\begin{equation}\label{eq: 2d example}
	\begin{aligned}
		\dot x&=-x^{n_x}+\alpha_x\tanh(x)-y+u_x\\
		\dot y&=-y^{n_y}+\alpha_y\tanh(y)-x+u_y
	\end{aligned}
\end{equation}
with $n_x,n_y\in\mathbb N$, $\alpha_x,\alpha_y,u_x,u_y\in\R$,
which is a monotone \RM{(because two-dimensional and competitive~\cite[Example~1]{smith1988systems}) dynamical system with parameters $\pi=(\alpha_x,\alpha_y,u_x,u_y)$.} Observe that the two exponents $n_x,n_y$ are not controllable parameters and fixed to $n_x=3$ and $n_y=5$ in what follows.

Figure~\ref{fig: 2d example phase p}({\bf a}), shows the phase portrait of model~\eqref{eq: 2d example} for nominal parameter values ($\pi=(3.0,4.0,0.3,1.0)$). Our goal is to \RM{ increase the basin of attraction of equilibrium $e_6$ (determined by the stable manifold branches highlighted in gray). We propose to do so by decreasing the real part of both eigenvalues $\lambda_1,\lambda_2$ of the linearization of model~\eqref{eq: 2d example} at equilibrium 6 by at least 1.5 units each.\footnote{This choice is for purely illustrative purposes.}}
We apply Control Strategy~\ref{CS: fully actu eig} \RM{with} MGDA~\cite{desideri2012multiple} to simultaneously descend \RM{along} the gradients associated to the two eigenvalues. The two control hyperparameters used are $\varepsilon=10^{-2}$ and $n_{\rm ite}=1000$. \RM{As an illustration, in the underactuated control case only variations along the two most sensitive parameters were allowed, with the other two components projected to zero.}

\begin{figure}
	\centering
	\includegraphics[width=0.475\textwidth]{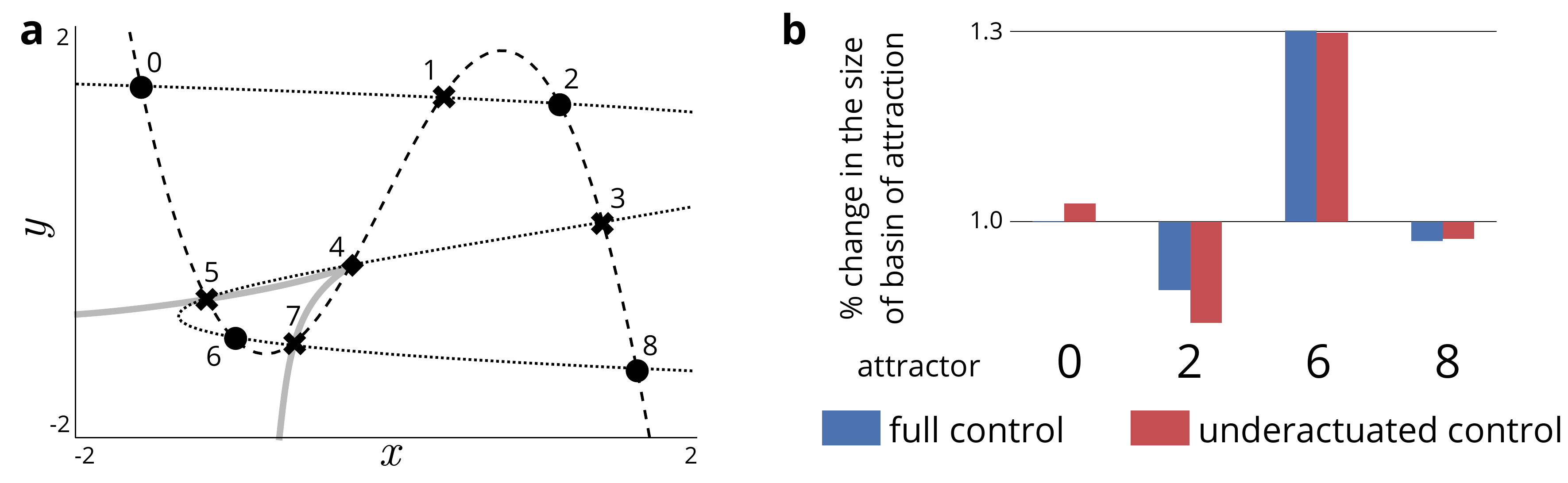}
	\caption{Phase portrait of model~\eqref{eq: 2d example} for nominal parameters ({\bf a}). Stable steady-states are marked with circles, saddle points with crosses, and unstable points with diamonds. $x$-nullcline is dashed. $y$-nullcline is dotted. \RM{ Saddle stable manifolds determining the boundaries of the basin of attraction of equilibrium $e_6$ are drawn in gray 
	 ({\bf b}). Change in the size of the basins of attraction of the four stable equilibria $e_0,e_2,e_6,e_8$ under full (blue) and underactued (red) control.
	}
	}\label{fig: 2d example phase p}
\end{figure}


\RM{Figure~\ref{fig: 2d example phase p}({\bf b}) shows the result of our control strategy after 544 iterations.
The size of the basin of attraction was estimated by randomly selecting initial conditions on the square $[-2, 2]\times[-2, 2]$ according to a uniform distribution and by recording the end state of each individual simulation. Simulation time was taken to be sufficiently long to ensure that practically steady-state was reached. The change in the size of the basin of attraction of each stable equilibrium $e_0,e_2,e_6,e_8$ was computed simply as the ratio between the number of initial conditions converging to it before and after control was applied. Both the full and the under-actuated control signal successfully increase the size of the basin of attraction of equilibrium $e_6$ by roughly $30\%$, while roughly maintaining or decreasing the size of the basin of attraction of the other stable equilibria. }

%


\section{Control of the EMT regulatory network}
\label{sec: MET control}

Multistable networks like molecule and gene regulatory networks are described using a variety of {\it qualitative} modeling approaches, like boolean networks~\cite{somogyi1996modeling} and piece-wise linear models~\cite{DeJong2004qualitative}. Data is indeed often unavailable to derive detailed quantitative models. The goal of qualitative models is mainly to capture binary molecular interactions and the transitions between different discrete cellular states. As suggested by~\cite{lenovere2015quantitative}, there is however a middle-ground between purely qualitative and detailed quantitative modeling, in which smooth ordinary differential equation (ODEs) models are not fitted to experimental data but, even so, variables and governing parameters remain quantitative. This allows a finer understanding of the model dynamics and, crucially, of the effects that variations of biologically relevant parameters have on them beyond all-or-none harsh manipulations like gene knock-in and knock-out.

\subsection{Model derivation}

Using boolean model reduction methods~\cite{matache2016logical,veliz2011reduction,saadatpour2013reduction}, it is possible to reduce the 9-dimensional boolean EMT model proposed in~\cite{mendez2017gene} to four boolean variables.
\footnote{\RM{The reduction is not driven by computational limitations of the method, which can be scaled nicely to high dimensions using, e.g., automatic differentiation tools. However, reducing the model dimension reduces the number of parameters to be identified/tuned for the nominal EMT dynamical behavior while still providing qualitative insights into the processes to be manipulated in an experimental setting, as discussed below.}}
 This results in a system that evolves according to the boolean difference equation~\eqref{eq:boolean_4d}, which preserves key regulatory genes and all the attractors of the full model (Table~\ref{subtab:equilibria_4d}).
\begin{table}
	\centering
{\small
		\caption{\footnotesize Four-dimensional reduction of the boolean model in~\cite{mendez2017gene} and associated boolean epithelial, senescent, and mesenchymal attractors. $\tau$ is the model discrete time. $\neg$ denotes the boolean NOT operation. $\land$ denotes the boolean AND operation. $\lor$ denotes the boolean OR operation. .}
		\label{subtab:equilibria_4d}
		\vspace{-3mm}
	{ \begin{equation}
		\label{eq:boolean_4d}
		\begin{split}
			S(\tau+1) &= \neg E(\tau) \lor \left( S(\tau) \land E(\tau) \land N(\tau) \right) \\
			E(\tau+1) &= \neg S(\tau) \lor \left( S(\tau) \land E(\tau) \land \neg N(\tau) \right) \\
			N(\tau+1) &= S(\tau) \lor E(\tau) \lor N(\tau) \lor P(\tau) \\
			P(\tau+1) &= \neg S(\tau) \land \left( \left( E(\tau) \land N(\tau) \right) \lor P(\tau) \right)
		\end{split}
	\end{equation}}
}
	{\scriptsize
	\begin{tabular}{c | c | c | c }
		Gene  (variable)               & Epithelial & Senescent & Mesenchymal \\
		\hline
		\textit{Snai2}  (S)     &     0      &     0     &      1       \\
		\textit{ESE2}    (E)    &     1      &     1     &      0       \\
		\textit{NF$\kappa$B} (N) &     1      &     1     &      1       \\
		\textit{p16}     (P)    &     0      &     1     &      0       \\
	\end{tabular}
}
\end{table}
Model~\eqref{eq:boolean_4d} can be translated to a parameterized system of differential equations by mapping boolean operators to sum and products of increasing or decreasing sigmoidal functions.  We use Hill functions $H(x,p,k)=\frac{x^{p}}{x^{p}+k^{p}}$ for increasing sigmoids and $\bar H(x,p,k)=1-H(x,p,k)$ for decreasing sigmoids. With these choices, we can map boolean operators to elementary algebraic operations between sigmoids:
\begin{align*}
	&\neg x\mapsto \bar H(x,p,k)\,,\quad  x\land y\mapsto H(x,p_x,k_x)H(y,p_y,k_y)\\
	&x\lor y\mapsto H(x,n_{x},k_{x})+H(y,n_{y},k_{y})
\end{align*}
Furthermore each interaction term is assumed to be parameterized by an interaction strength $\alpha$, and each variable has a {\em linear degradation} term $-x$ and a constant {\em source term} $\beta$. The proposed translation from boolean to smooth ODEs is similar in spirit to~\cite{mendoza_method_2006} but with the crucial difference that interaction strengths and half-activations are parameterized, and that the smooth variables are not assumed to live in a unitary hypercube. The resulting quantitative dynamics are
\begin{align}\label{eq:ode_4d}
	\tfrac{dS}{dt} =& \alpha_1 \tfrac{k_1^p}{E^p + k_1^p} + \alpha_2 \tfrac{S^p}{S^p + k_2^p}  \tfrac{E^p}{E^p + k_3^p}  \tfrac{N^p}{N^p + k_4^p} + \beta_S - S \nonumber \\
	\tfrac{dE}{dt} = &\alpha_3 \tfrac{k_5^p}{S^p + k_5^p} + \alpha_4 \tfrac{E^p}{E^p + k_6^p}  \tfrac{S^p}{S^p + k_7^p} \tfrac{k_8^p}{N^p + k_8^p}+ \beta_E - E \nonumber\\
	\tfrac{dN}{dt} =& \alpha_5 \tfrac{S^p}{S^p + k_9^p} + \alpha_6 \tfrac{E^p}{E^p + k_{10}^p} + \alpha_7 \tfrac{N^p}{N^p + k_{11}^p} + \alpha_8 \tfrac{P^p}{P^p + k_{12}^p} \nonumber\\& + \beta_N - N \\
	\tfrac{dP}{dt} =& \alpha_9 \tfrac{k_{13}^p}{S^p + k_{13}^p} \left[ \alpha_{10} \tfrac{E^p}{E^p + k_{14}^p} \tfrac{N^p}{N^p + k_{15}^p} + \alpha_{11} \tfrac{P^p}{P^p + k_{16}^p} \right]\nonumber \\&+ \beta_P - P \nonumber
\end{align}
\RM{Each lumped variable ($S,E,N,P$) is associated to a whole functional module of the actual EMT regulatory network.}
In the limit $p\to\infty$, Hill functions become binary activation functions and model~\eqref{eq:ode_4d} becomes piece-wise linear, as in~\cite{DeJong2004qualitative}.

\subsection{Control through local sensitivity}

For the nominal set of parameters $\pi$ there are 3 stable points with purely real negative eigenvalues corresponding to the epithelial, senescent, and mesenchymal phenotypes, 3 saddle points separating their basins of attraction, and 1 unstable point.
\RM{
To verify that Assumption~\ref{ass: separatrices} holds for~\eqref{eq:ode_4d} with nominal parameters we numerically approximated the saddle unstable manifolds by picking initial conditions in a small neighborhood of each saddle and letting the trajectory converge. The resulting heteroclinic orbits were found to be monotone and revealed the presence of one saddle between the epithelial and senescent equilibria and of two saddles between the epithelial and mesenchymal equilibria. The epithelial and mesenchymal equilibria do not share boundaries of their basins of attraction. A Monte Carlo on initial conditions revealed that no other periodic or `strange' attractors or separatrices existed. Finally, because the linear part of~\eqref{eq:ode_4d} is exponentially stable and the nonlinear part is bounded, it is easy to show that its trajectories are bounded and Assumption~\ref{ass: separatrices} holds.}

%
%
Our goal is to increase the size of the basin of attraction of the epithelial equilibrium $x^*_e$. We do so by computing through MGDA the sensitive direction in the parameter space that simultaneously {\it i)} makes more negative the two eigenvalues of $J^\pi_{x^*_e}$ whose eigenvectors are tangent to the unstable manifolds of the saddle points surrounding $x^*_e$ 
and {\it ii)} increase the average distance $\overline{ g_{x^*_e-s_{12}}}$ between $x^*_e$ and the \RM{two} saddle points delimiting its basin of attraction. The stopping criterion is that all stable equilibria are maintained, i.e., none of them disappear in a bifurcation, \RM{or the maximum number of iterations $n_{\rm ite}=1000$ is reached}.

\begin{figure}
	\centering
	\includegraphics[width=0.475\textwidth]{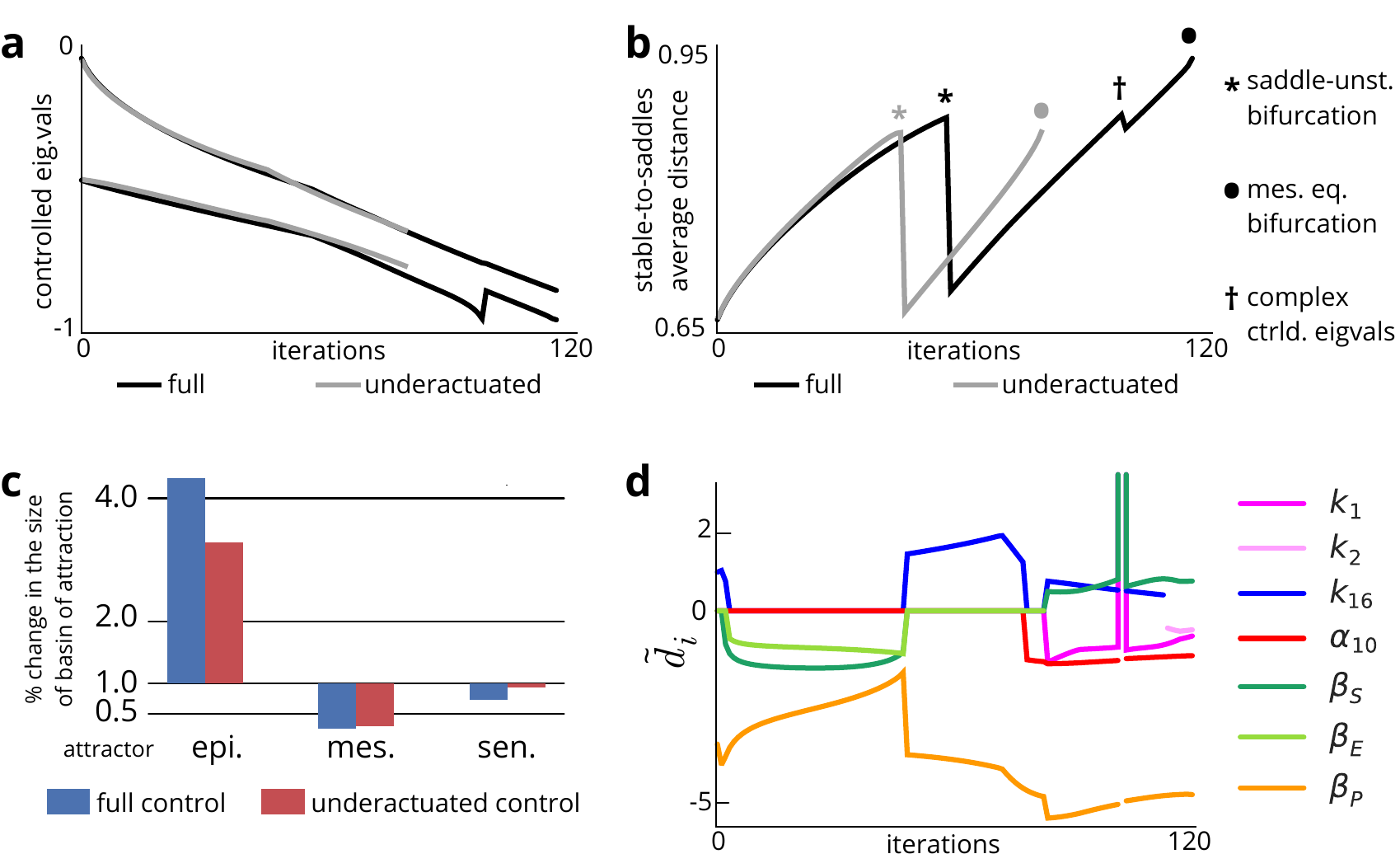}
	\caption{{\bf a}. Evolution of the two controlled eigenvalues across control iterations. {\bf b}. Evolution of the average distance between $x^*_e$ and saddle points delimiting its basin of attraction. {\bf c}. Change in the size of the basin of attraction of the three stable equilibria after either the \RM{full} control (blue) or the \RM{underactuated} control (red) are applied. {\bf d}. Components of the sensitive parameter control direction associated to the most sensitive parameters. $\bullet$: bifurcation of the mesenchymal stable equilibrium. $\ast$: bifurcation of one of the saddle\RM{s} delimiting the basin of attraction of $x^*_e$ with an\RM{other} unstable equilibrium. $\dagger$:  non-zero imaginary parts of the controlled eigenvalues.}\label{fig: 4D control}
\end{figure}

The effects of the applied control strategy are summarized in Figure~\ref{fig: 4D control}. Both controlled eigenvalues (Figure~\ref{fig: 4D control}a) and $\overline{ g_{x^*_e-s_{123}}}$ (Figure~\ref{fig: 4D control}b) change in the desired direction across control iterations. As a consequence and as predicted by our theory, the size of the basin of attraction\footnote{Computed through a Monte Carlo on initial conditions in {$[0, 4]^4$}.} of $x^*_e$  is increased \RM{four-fold} (Figure~\ref{fig: 4D control}c), that of the senescent equilibrium shrinks slightly, {and that of the mesenchymal equilibrium is reduced to approximately a third of its original size}. The abrupt drops in the evolution of $\overline{ g_{x^*_e-s_{12}}}$ correspond either to bifurcations of the model unstable equilibria or to some eigenvalues becoming complex \RM{at a stable equilibrium, which transiently violates Assumption~\ref{ass: separatrices}} (see Figure~\ref{fig: 4D control} for details).

To test the robustness and practical applicability of our control strategy, we assumed that only three out of 32 parameters with the largest absolute sensitivity in $\tilde d$ could be manipulated at each control iteration (while other components are projected to zero). This is again a kind of underactuated control. As summarized in Figures~\ref{fig: 4D control}a,b,c, although a slight drop in performance can be detected, the control goal is robustly achieved. 

Tracking the evolution of the components \RM{of the sensitive parameter control direction (Figures~\ref{fig: 4D control}d) and, in particular, of the largest three used in the underactuated control strategy,}  reveals the key parameters to be manipulated and whether they should be increased or decreased. For instance, initially the three source terms ($\beta_i$, $i=S,E,P$) must all be decreased. Subsequently, the half-activation $k_{16}$ must be increased and interaction gain  $\alpha_{10}$ decreased.
%
\RM{Due to model reduction, these parameters must be interpreted as lumped parameters corresponding to whole regulation pathways between four modules (associated to the lumped variables $S,E,N,P$) of the actual EMT regulatory network. The lumped parameter manipulations suggested by our method provide a qualitative guide of which molecular processes could be manipulated to achieve the same control objective experimentally.}

\section{Conclusions}
\label{sec: conclusions}

We derived a parameter control law to control the size and depth of the basins of attraction of multistable dynamics with simple attractors~\RM{and simple separatrices. Monotone dynamical systems are important representatives of the class of models to which our method applies.}. Our control law is cheap to compute \RM{because it solely uses local, i.e., linearized, information of the model dynamics at its equilibria.} When applied to biological models,  our approach is able to suggest counter-intuitive parameter manipulations that can be tested in experimental settings to achieve desired control objectives. We illustrated this fact on the control of a  new ODE model of the EMT, with relevance for the control of tumor formation and metastasis.

\RM{The main drawback of the proposed methodology is that it is grounded on a heuristic argument. It is therefore difficult to provide theoretical guarantees. For instance, the boundaries of the relevant basins of attraction might exhibit complicated shapes that are completely unpredictable by the local methods employed here. Despite enforcing both eigenvalues and saddle control, the boundaries of the basins of attraction might bend in such a way that the proposed control strategy would turn out disruptive for the control objective. An implicit claim underlying this work is that {\it this cannot happen for monotone systems} but this still needs to be rigorously proved. Furthermore, beyond monotone dynamics, Assumption~\ref{ass: separatrices} must hold for our method to work. We were able to verify this assumption numerically for our EMT four-dimensional model but achieving the same in other models might not always be easy.}

\section{Code availability}\label{SEC: code}

The code used to run the paper simulations and generate the related figures, including the sets of used parameters, is available at GitHub: \url{https://github.com/rodrigo-moreno/basin-control-notebook}.


\bibliographystyle{IEEEtran}
\bibliography{biblio_V2}

\begin{thebibliography}{10}
\providecommand{\url}[1]{#1}
\csname url@samestyle\endcsname
\providecommand{\newblock}{\relax}
\providecommand{\bibinfo}[2]{#2}
\providecommand{\BIBentrySTDinterwordspacing}{\spaceskip=0pt\relax}
\providecommand{\BIBentryALTinterwordstretchfactor}{4}
\providecommand{\BIBentryALTinterwordspacing}{\spaceskip=\fontdimen2\font plus
\BIBentryALTinterwordstretchfactor\fontdimen3\font minus
  \fontdimen4\font\relax}
\providecommand{\BIBforeignlanguage}[2]{{%
\expandafter\ifx\csname l@#1\endcsname\relax
\typeout{** WARNING: IEEEtran.bst: No hyphenation pattern has been}%
\typeout{** loaded for the language `#1'. Using the pattern for}%
\typeout{** the default language instead.}%
\else
\language=\csname l@#1\endcsname
\fi
#2}}
\providecommand{\BIBdecl}{\relax}
\BIBdecl

\bibitem{pisarchik2014control}
A.~N. Pisarchik and U.~Feudel, ``Control of multistability,'' \emph{Physics
  Reports}, vol. 540, no.~4, pp. 167--218, 2014.

\bibitem{bizyaeva2022}
A.~Bizyaeva, A.~Franci, and E.~L. Naomi, ``Nonlinear opinion dynamics with
  tunable sensitivity,'' \emph{To appear in IEEE Trans. Aut. Contr.}, 2022.

\bibitem{balazsi2011cellular}
G.~Bal{\'a}zsi, A.~Van~Oudenaarden, and J.~J. Collins, ``Cellular decision
  making and biological noise: from microbes to mammals,'' \emph{Cell}, vol.
  144, no.~6, pp. 910--925, 2011.

\bibitem{nieto_emt_2016}
\BIBentryALTinterwordspacing
M.~A. Nieto, R.~Y.-J. Huang, R.~A. Jackson, and J.~P. Thiery, ``{EMT}: 2016,''
  \emph{Cell}, vol. 166, no.~1, pp. 21--45, Jun. 2016, publisher: Elsevier.
  [Online]. Available: \url{https://doi.org/10.1016/j.cell.2016.06.028}
\BIBentrySTDinterwordspacing

\bibitem{yang_epithelial-mesenchymal_2008}
\BIBentryALTinterwordspacing
J.~Yang and R.~A. Weinberg, ``Epithelial-{Mesenchymal} {Transition}: {At} the
  {Crossroads} of {Development} and {Tumor} {Metastasis},'' \emph{Developmental
  Cell}, vol.~14, no.~6, pp. 818--829, Jun. 2008, publisher: Elsevier.
  [Online]. Available: \url{https://doi.org/10.1016/j.devcel.2008.05.009}
\BIBentrySTDinterwordspacing

\bibitem{chakrabarti_elf5_2012}
\BIBentryALTinterwordspacing
R.~Chakrabarti, J.~Hwang, M.~Andres~Blanco, Y.~Wei, M.~Lukačišin, R.-A.
  Romano, K.~Smalley, S.~Liu, Q.~Yang, T.~Ibrahim, L.~Mercatali, D.~Amadori,
  B.~G. Haffty, S.~Sinha, and Y.~Kang, ``Elf5 inhibits the
  epithelial–mesenchymal transition in mammary gland development and breast
  cancer metastasis by transcriptionally repressing {Snail2},'' \emph{Nature
  Cell Biology}, vol.~14, no.~11, pp. 1212--1222, Nov. 2012. [Online].
  Available: \url{https://doi.org/10.1038/ncb2607}
\BIBentrySTDinterwordspacing

\bibitem{efimov2012global}
D.~Efimov, ``Global lyapunov analysis of multistable nonlinear systems,''
  \emph{SIAM Journal on Control and Optimization}, vol.~50, no.~5, pp.
  3132--3154, 2012.

\bibitem{smith1988systems}
H.~L. Smith, ``Systems of ordinary differential equations which generate an
  order preserving flow. a survey of results,'' \emph{SIAM review}, vol.~30,
  no.~1, pp. 87--113, 1988.

\bibitem{desideri2012multiple}
J.-A. D{\'e}sid{\'e}ri, ``Multiple-gradient descent algorithm ({MGDA}) for
  multiobjective optimization,'' \emph{Comptes Rendus Mathematique}, vol. 350,
  no. 5-6, pp. 313--318, 2012.

\bibitem{somogyi1996modeling}
R.~Somogyi and C.~A. Sniegoski, ``Modeling the complexity of genetic networks:
  understanding multigenic and pleiotropic regulation,'' \emph{complexity},
  vol.~1, no.~6, pp. 45--63, 1996.

\bibitem{DeJong2004qualitative}
H.~De~Jong, J.-L. Gouz{\'e}, C.~Hernandez, M.~Page, T.~Sari, and J.~Geiselmann,
  ``Qualitative simulation of genetic regulatory networks using
  piecewise-linear models,'' \emph{Bulletin of mathematical biology}, vol.~66,
  no.~2, pp. 301--340, 2004.

\bibitem{lenovere2015quantitative}
N.~Le~Novere, ``Quantitative and logic modelling of molecular and gene
  networks,'' \emph{Nature Reviews Genetics}, vol.~16, no.~3, pp. 146--158,
  2015.

\bibitem{matache2016logical}
M.~T. Matache and V.~Matache, ``Logical reduction of biological networks to
  their most determinative components,'' \emph{Bulletin of mathematical
  biology}, vol.~78, no.~7, pp. 1520--1545, 2016.

\bibitem{veliz2011reduction}
A.~Veliz-Cuba, ``Reduction of {Boolean} network models,'' \emph{Journal of
  theoretical biology}, vol. 289, pp. 167--172, 2011.

\bibitem{saadatpour2013reduction}
A.~Saadatpour, R.~Albert, and T.~C. Reluga, ``A reduction method for {Boolean}
  network models proven to conserve attractors,'' \emph{SIAM Journal on Applied
  Dynamical Systems}, vol.~12, no.~4, pp. 1997--2011, 2013.

\bibitem{mendez2017gene}
L.~F. M{\'e}ndez-L{\'o}pez, J.~Davila-Velderrain,
  E.~Dom{\'\i}nguez-H{\"u}ttinger, C.~Enr{\'\i}quez-Olgu{\'\i}n, J.~C.
  Mart{\'\i}nez-Garc{\'\i}a, and E.~R. Alvarez-Buylla, ``Gene regulatory
  network underlying the immortalization of epithelial cells,'' \emph{BMC
  systems biology}, vol.~11, no.~1, pp. 1--15, 2017.

\bibitem{mendoza_method_2006}
L.~Mendoza and I.~Xenarios, ``A method for the generation of standardized
  qualitative dynamical systems of regulatory networks,'' \emph{Theoretical
  Biology and Medical Modelling}, vol.~3, no.~1, p.~13, 2006.

\end{thebibliography}

\end{document}